\documentclass[12pt, a4paper, reqno, oneside]{amsart}
\usepackage{amssymb}  
\usepackage{amsmath} 
\usepackage[normalem]{ulem}
\usepackage{easyReview}
\usepackage{amsthm}  
\usepackage{graphicx}
\usepackage{bm}
\usepackage{color}
\usepackage{diagbox}
\usepackage{listings}
\usepackage{longtable}
\usepackage{mathrsfs}
\usepackage{caption}
\usepackage[T2A]{fontenc}
\usepackage[utf8]{inputenc}
\usepackage{cite, verbatim}
\usepackage{multirow}
\usepackage{subfigure}

\newtheorem{theorem}{Theorem}
\newtheorem*{theorem*}{Theorem}

\newtheorem{lemma}{Lemma}
\numberwithin{lemma}{section}

\newtheorem{cor}{Corollary}
\newtheorem{conj}{Conjecture}

\theoremstyle{definition}
\newtheorem*{rem}{Remark}

\numberwithin{equation}{section}

\newcommand{\Aut}{\operatorname{Aut}}

\newcommand{\pMod}[1]{\,(\mathrm{mod}\ #1)}

\begin{document}

\vspace{1cm}

\title[\textsc{On  recognition  of simple groups with  disconnected prime graph by spectrum
}]{\textsc{On  recognition  of simple groups with  disconnected prime graph by spectrum
}}

\author{\textsc{V.V. Panshin}}
\address{Novosibirsk State University}
\email{v.pansh1n@yandex.ru}

\thanks{The work is supported by the Russian Science Foundation, Project 24–11–00127, https://rscf.ru/en/project/24-11-00127/}

\begin{abstract}
Two finite groups are said to be isospectral if they have the same sets of element orders. This paper completes the description of finite groups that are isospectral to finite simple groups with disconnected prime graph.

\end{abstract}

\begingroup
\def\uppercasenonmath#1{} 
\let\MakeUppercase\relax 
\maketitle
\endgroup

\section{Introduction}

The spectrum $\omega(G)$ of a finite group $G$ is the set of element orders of $G$. Groups whose spectra coincide are said to be \textit{isospectral}. We say that \textit{the problem of recognition by spectrum is solved} for $G$ if the number $h(G)$ of distinct (pairwise non-isomorphic) finite groups isospectral to $G$ is known, and if $h(G)$ is finite, then all such groups are explicitly described. If $h(G)=1$, then $G$ is called \textit{recognizable by spectrum}; if $h(G)=\infty$, then it is called \textit{unrecognizable}. A survey of the current state of the recognition by spectrum problem can be found in \cite{22Survey}.

By~\cite[Lemma 1]{98Maz.t}, if $G$ has a nontrivial normal solvable subgroup, then $h(G)=\infty$, therefore the problem of recognition by spectrum is of interest for groups with trivial solvable radical and, primarily, for nonabelian simple groups. At present, the problem of recognition by spectrum is solved for all nonabelian simple groups, except for the groups from the following list $\mathcal{L}$ (in all cases $q$ is odd, and simple classical groups are denoted according to~\cite{85Atlas}):
\begin{enumerate} 
\item[(a)] $L_n(q)$, with $8\leq n \leq 26$, $n$ is not a prime;
\item[(b)] $U_n(q)$, with $7\leq n \leq 26$;
\item[(c)] $S_{2n}(q)$ and $O_{2n+1}(q)$, with $5\leq n \leq 15$, $n\neq8$;
\item[(d)] $O^+_{2n}(q)$, with $5\leq n \leq 18$;
\item[(e)] $O^-_{2n}(q)$, with $5\leq n \leq 17$, $n\neq8,16$.   \end{enumerate}

If $L$ is not in the list $\mathcal{L}$ and $L\neq U_5(q)$, $L_6(q)$, $U_6(q)$, then groups isospectral to $L$ can be found in \cite[Theorem 2.1]{22Survey}. The groups $U_5(q)$, $L_6(q)$, and $U_6(q)$ were considered in \cite{24GrePan.t}. It should be noted that for some groups in $\mathcal{L}$, the recognition problem is also solved, but these groups are not explicitly mentioned to keep the list $\mathcal{L}$ compact.

We call a finite nonabelian simple group $L$ \textit{almost recognizable by spectrum} if any finite group $G$ isospectral to $L$ is an almost simple group with socle isomorphic to $L$, i.e., up to isomorphism, we have $L\leq G\leq \Aut L$. Clearly, for an almost recognizable group $L$, the recognition problem is reduced to describing all almost simple groups with socle $L$ that are isospectral to $L$. Since for any nonabelian simple group $L$, the almost simple groups with socle $L$ that are isospectral to $L$ are known (see \cite[Theorem 3.7]{22Survey}), if a simple group is almost recognizable, then the problem of recognition by spectrum is solved for this group.

There is a conjecture \cite[Conjecture 3.10]{22Survey} that all groups in $\mathcal{L}$ are almost recognizable by spectrum. The proof of this conjecture is reduced to a certain special case: if $L\in \mathcal{L}$ and $G$ is a finite group isospectral to $L$, then by virtue of~\cite[Theorems 3.1, 3.6, and 3.8]{22Survey}, either $G$ is an almost simple group with socle isomorphic to $L$, or $G$ has a unique nonabelian composition factor $S$ and $S$ is a group of Lie type over a field whose characteristic is different from the defining characteristic of $L$. Thus, the conjecture about almost recognizability of all groups in $\mathcal{L}$ is equivalent to the following conjecture.

\begin{conj}\label{conj1} Let $L\in\mathcal{L}$, let $S$ be a finite simple group of Lie type over a field whose characteristic differs from the defining characteristic of $L$, and let $G$ be a finite group with $\omega(G)=\omega(L)$. Then $S$ is not a composition factor of $G$.
\end{conj}

The main goal of this paper is to solve the recognition problem for all groups in the list $\mathcal L$ whose prime graph is disconnected. The \textit{prime graph} (or \textit{Gruenberg–Kegel graph}) $GK(G)$ of a group $G$ is the graph whose vertices are the prime divisors of the order of $G$, with two distinct vertices $r$ and $s$ adjacent if and only if $rs\in\omega(G)$. The simple groups with disconnected prime graph were found by Williams~\cite{81Wil} and Kondrat'ev~\cite{89Kon.t}; corrected tables listing these groups can be found, for example, in~\cite{02Maz.t}. 

Now we give a brief overview of the results concerning the recognition problem for simple groups with disconnected prime graph. All alternating groups with disconnected prime graph, except $A_6$, are recognizable by spectrum (see \cite{00Zav.t}, \cite{00KonMaz.t}). The group $A_{6}$ is unrecognizable~\cite{91BrShi}. All sporadic simple groups have disconnected prime graph, and all of them, except $J_2$, are recognizable by spectrum, while $J_2$ is unrecognizable (see~\cite{98MazShi} and the references therein).

All exceptional groups of Lie type, except for the groups $E_7(q)$ with $q > 3$, have  disconnected prime graph. In \cite{13Maz.t}, it was shown that ${}^3D_4(2)$ is unrecognizable. Recognizability of the Ree and Suzuki groups was established in \cite{92Shi, 93BrShi, 99DenShi}, while for the groups $G_2(q)$ and $E_8(q)$, it was established in \cite{02Vas.t, 13VasSt.t, 10Kon}. Moreover, the groups $E_7(2)$ and $E_7(3)$ are recognizable by the prime graph \cite{14Kon.r}. The solution to the recognition problem for the groups ${}^3D_4(q)$, with $q \neq 2$, $F_4(q)$, $E_6(q)$, and ${}^2E_6(q)$, was obtained in a series of works: almost recognizability was proved in \cite{05AlKon.t, 06Al.t, 07Kon.t, 15Gr}, isospectral almost simple groups were found in \cite{16GrZv, 16Zve.t}. The final results are summarized in \cite[Table 1]{16Zve.t}.

The classical simple groups with disconnected prime graph are listed in Table \ref{disconnected_graph_table} in Section~\ref{s:prelim}. The solution to the recognition problem for the groups $L_2(q)$, $L_3(q)$, $U_3(q)$, and $S_4(q)$ can be found in \cite{94BrShi}, \cite{04Zav1.t}, \cite{06Zav.t}, and \cite{02Maz.t}. The groups $S_8(q)$ with $q \neq 7^{m}$ and $O_9(q)$ are unrecognizable, while $S_8(7^m)$ are recognizable \cite{14GrSt, 25Gre.r}.

The recognition problem for linear and unitary groups in characteristic $2$ was solved in \cite{08VasGr.t, 13GrShi}. All orthogonal and symplectic groups in characteristic $2$, except for the groups $S_4(q)$, $S_6(2)$, $O^+_8(2)$, and $S_8(q)$, are recognizable (see \cite{15VasGr.t}). All groups $O_{2n}^-(q)$ with disconnected prime graph and odd $n$ and $q$ are recognizable \cite{08AlKon, 09Kon}, as well as the groups $O_{2r+1}(3)$, where $r \geq 5$ is prime, and $S_{2r}(3)$, where $r \geq 3$ is prime, \cite{10Zin, 10SheShiZin.t}. It is also known that $\omega(S_6(2)) = \omega(O_8^+(2))$, $h(S_6(2)) = 2$ \cite{97Maz.t}, and $\omega(O_7(3)) = \omega(O_8^+(3))$, $h(O_7(3))=2$ \cite{97ShiTan}.

The almost recognizability of the remaining symplectic and orthogonal groups with disconnected prime graph has been established in \cite{09AleKon.t, 09HeShi, 09Kon1, 09VasGorGr.t, 15Gr}.
The almost recognizability of $L_r(q)$, where $r \ge 5$ is prime, was proved in \cite{12GrLyt.t}. Together with the description of almost simple classical groups that are isospectral to their socle from \cite{16Gr.t, 17Gre, 18Gr.t}, these results give the solution for the recognition problem for these groups.

Thus, for groups with disconnected prime graph in $\mathcal{L}$, the recognition problem remains unresolved only for the groups $U_n(q)$, where $n \in \{7,11,13,17,19,23\}$, and $L^\varepsilon_{n}(q)$, where $n \in \{8,12,14,18,20,24\}$ and $q-\varepsilon$ divides $n$ (hereafter, $L^{+}_n(q) = L_n(q)$ and $L^{-}_n(q) = U_n(q)$). We prove Conjecture 1 for these groups and, as a consequence, establish that they are almost recognizable by spectrum.

\begin{theorem}\label{t:t1}

Let $L$ be a linear or unitary group in the list $\mathcal{L}$ with disconnected prime graph, and let $G$ be a finite group such that $\omega(G)=\omega(L)$. If $S$ is a nonabelian composition factor of $G$, then $S$ is not a group of Lie type over a field whose characteristic differs from the defining characteristic of $L$.
\end{theorem}

\begin{cor}\label{c:t4} 
Let $L$ be a linear or unitary group in $\mathcal{L}$ with disconnected prime graph, and let $G$ be a finite group such that $\omega(G)=\omega(L)$. Then $G$ is an almost simple group with socle isomorphic to $L$.
\end{cor}

As previously noted, having almost recognizability, we know the solution to the recognition problem. Thus, Corollary~\ref{c:t4} and \cite[Theorems 2 and 3]{17Gre} provide a complete description of finite groups isospectral to linear and unitary groups with disconnected prime graph in the list $\mathcal{L}$. Since, as shown above, for all other finite simple groups $L$ with disconnected prime graph, the number $h(L)$ was found and in the cases $h(L)<\infty$, all finite groups isospectral to $L$ are described, the following general result holds.

\begin{theorem}\label{t:main}
The problem of recognition by spectrum is solved for all finite simple groups with disconnected prime graph.
\end{theorem}

\section{Preliminaries}\label{s:prelim}

Let $a$ and $b$ be integers, and $r$ a prime. The greatest common divisor and least common multiple of $a$ and $b$ are denoted by $(a, b)$ and $[a, b]$, respectively. The set of all prime divisors of $a$ is denoted by $\pi(a)$. If $\pi$ is a set of primes, then $(a)_\pi$ is the $\pi$-part of $a$, i.e., the largest divisor of $a$, all of whose prime divisors are in $\pi$. In particular, we write $(a)_r$ and $(a)_{r'}$ when $\pi=\{r\}$ or $\pi$ is the set of all primes different from $r$, respectively. The following statement is well known (see, for example, \cite[Chapter IX, Lemma 8.1]{82HupBl2}).

\begin{lemma}\label{l:r-part}
Let $a$ be an integer, with $|a| > 1$, and let $m$ be a natural number. If $r$ is an odd prime and $a \equiv 1 \pMod r$, then $(a^m - 1)_r = (m)_r(a - 1)_r$.
\end{lemma}

If a prime $r$ is odd and does not divide $a$, then $e(r,a)$ denotes the multiplicative order of $a$ modulo $r$. For an odd integer $a$, set $e(2, a) = 1$ if $a\equiv 1\pMod 4$, and $e(2, a) =
2$ if $a\equiv 3\pMod 4$. Let $|a|>1$ and $i\geq 1$. A prime $r$ is called a \emph{primitive prime divisor} of $a^i-1$ if $e(r, a) = i$. The set of all primitive prime divisors of $a^i-1$ is denoted by $R_i(a)$, and any element of $R_i(a)$ is denoted by $r_i(a)$. The existence of primitive prime divisors for all but finitely many pairs $(a,i)$ was proved by Bang \cite{86Bang} and Zsigmondy \cite{Zs}.

\begin{lemma}[Bang--Zsigmondy]
Assume that $a$ is an integer with $|a|>1$. Then for any positive integer $i$ the set $R_i(a)$ is nonempty, except in the cases when $(a,i)\in\{(2,1),(2,6),$ $(-2,2),$$(-2,3),(3,1),(-3,2)\}$.
\end{lemma}

The following property of primitive divisors follows from Fermat's little theorem.

\begin{lemma}\label{ri_structure}
	Let $r\in R_i(a)$. Then $r=ik+1$, $k\geq1$.
\end{lemma}

For $i\geq 3$, the greatest primitive divisor $k_i(a)$ of the number $a^i-1$ is defined as the product of $(a^i-1)_r$ over all $r\in R_i(a)$ (we will not consider the greatest primitive divisors $k_1(q)$ and $k_2(q)$). As usual, $\varphi(x)$ and $\Phi_i(x)$ denote Euler's totient function and the $i$-th cyclotomic polynomial, respectively. Let $F(m)$ denote the sum $\sum_{i=1}^{m}\varphi(i)$.

\begin{lemma}\label{l:23}
If $3\leq m\leq 24$, then $\prod_{i=1}^{m}\Phi_i(\pm u)>u^{F(m)-2}$ for every $u\geq 2$.
\end{lemma}
\begin{proof}
The inequality is straightforward to verify.
\end{proof}

\begin{lemma}\label{l:kiPhi}
Let $a$ and $i$ be integers, $|a|\geq 2$ and $i\geq 3$.
Let $r$ be the greatest prime divisor of $i$ and $l = (i)_{r'}$. Then  
$$k_i(a) = \dfrac{|\Phi_i(a)|}{(r,\Phi_l(a))}.$$
Moreover, if $l$ does not divide $r-1$, then $(r, \Phi_l(a)) = 1$.
\end{lemma}
\begin{proof}
See~\cite[Proposition 2]{97Roi}. 
\end{proof}

Denote by $\pi(G)$ the set of prime divisors of the order of $G$. The set of element orders of $G$ that are maximal under divisibility is denoted by $\mu(G)$. The exponent of $G$ is denoted by $\exp(G)$. The spectra of finite simple groups of Lie type are known, and references to the corresponding papers can be found in \cite{18But.t} (see also \cite[Lemma 2.3]{16Gr.t} for misprints in the description of the spectra of $O_{2n}^\pm(q)$).

\begin{lemma}\label{l:exponents}
Let $S$ be a nonabelian simple classical group over the field of characteristic $v$ and order $u$. 

\begin{enumerate}
    \item[(a)] Let $S = L_n^{\varepsilon}(u)$, where $\varepsilon \in \{+,-\}$, $n \geq 3$, and let $v^l$ be the smallest power of $v$ greater than $n-1$. Put $c = r$ if $r \in \pi(u - \varepsilon)$ and $n = r^s$, and $c = 1$ otherwise. Then
    \[
        \exp(S) = \frac{v^l}{c} \prod_{i=1}^n \Phi_i(\varepsilon u) > u^{3F(n)/4}.
    \]
    \item[(b)] Let $S$ be one of $S_{2n}(u)$, $O_{2n+1}(u)$, where $n \geq 2$, or $O_{2n+2}^+(u)$, where $n \geq 3$ is odd, or $O_{2n}^-(u)$, where $n \geq 4$ is even. Let $v^l$ be the smallest power of $v$ greater than $2n-3$ in the latter case and $2n-1$ otherwise. Put $c = (2, u-1)^2$ if $n = 2^s$ and $c = (2, u-1)$ if  $n \ne 2^s$. Then
\[
\exp(S) = \frac{v^l}{c} \prod_{i=1}^n \Phi_i(u^2) > u^{3F(n)/2}.
\]

\item[(c)] Let $S = O_{2n}^{\varepsilon}(u)$, where $\varepsilon \in \{+,-\}$, $n \geq 5$ is odd, and let $v^l$ be the smallest power of $v$ greater than $2n-3$. Then
\[
\exp(S) = \frac{v^l}{(2, u-1)} \Phi_n(\varepsilon u) \prod_{i=1}^{n-1} \Phi_i(u^2) > u^{(3F(n) + 3F(n-1))/4}.
\]
\end{enumerate}
\end{lemma}
\begin{proof}
	See~\cite[Lemma 3.5]{19GrVasZv}.
\end{proof}

\begin{lemma}\label{L_spectra}
	Let $S=L^\varepsilon_m(q)$, where $m\geq2$, let $q$ be a power of a prime $p$. Put $d=(m,q-\varepsilon)$. Then $\omega(S)$ is exactly the set of divisors of the following numbers:
	\begin{enumerate}
		\item[(a)]  $\dfrac{q^m-\varepsilon^m}{d(q-\varepsilon)}$;
		\item[(b)]  $\dfrac{[q^{m_1}-\tau^{m_1},q^{m_2}-\tau^{m_2}]}{(m/(m_1,m_2),q-\tau)}$ for $m_1,m_2>0$ such that $m_1+m_2=m$;
		\item[(c)] $[q^{m_1}-\varepsilon^{m_1},\dots,q^{m_s}-\varepsilon^{m_s}]$ for $s\geq3$ and $m_1,\dots,m_s>0$ such that $m_1+\dots+m_s=m$;
		\item[(d)] $p^k\dfrac{q^{m_1}-\varepsilon^{m_1}}{d}$ for $k,m_1>0$ such that $p^{k-1}+1+m_1=m$;
		\item[(e)] $p^k[q^{m_1}-\varepsilon^{m_1},\dots,q^{m_s}-\varepsilon^{m_s}]$ for $s\geq2$ and $k,m_1,\dots,m_s>0$ such that $p^{k-1}+1+m_1+\dots+m_s=m$;
		\item[(f)] $p^k$, if $p^{k-1}+1=m$ for $k>0$.
	\end{enumerate}
\end{lemma}
\begin{proof}
See \cite[Corollary 3]{08But.t}.    
\end{proof}

\begin{lemma}\label{MAGr_lemma} Let $n$ be a prime, $n\geqslant 5$, let $q$ be a power of an odd prime $p$ and $a=\frac{q^n+1}{(q+1)(n,q+1)}-1$. Then the following statements hold:
	
	\begin{enumerate}
		\item[(a)] if $n$ divides $q+1$, then $(q+1)/n$  divides $a$  and $(a)_n=(q+1)_n$;
		\item[(b)] $a\not\in\omega(U_n(q))$.
	\end{enumerate}
	
\end{lemma}

\begin{proof}

        If $n$ does not divide $q+1$, then $a$ is divisible by $pr_{n-1}(-q)$, hence $a\notin\omega(U_n(q))$ by Lemma~\ref{L_spectra}. From now on assume that $n$ divides $q+1$.

        Expanding $q^n$ in powers of $q+1$ and then computing $a$, we obtain
    $$a=\frac{q+1}{n}\cdot \sum_{i=2}^{n}C_n^i (q+1)^{i-2}(-1)^{n-i}.$$ 
    The sum on the right is congruent to $-n(n-1)/2$ modulo $n^2$, hence $(a)_n=(q+1)_n$.

        Suppose that $a\in\omega(U_n(q))$. Then $a$ divides one of the numbers listed in Lemma~\ref{L_spectra}. Since $(a)_n=(q+1)_n$, it cannot divide the numbers in (a), (b), (d), or (e).
	
        We now show that the numbers in the remaining items are less than $a$.
	
	Let $b=[q^{m_1}-(-1)^{m_1},\dots,q^{m_s}-(-1)^{m_s}]$, where $m_1+\dots+m_s=n$ and $s\geq 3$. Then $$b\leq (q+1)\cdot \frac{q^{m_1}-(-1)^{m_1}}{q+1}\dots\frac{q^{m_s}-(-1)^{m_s}}{q+1}<(q+1)q^{n-s}\leq (q+1)q^{n-3}.$$ 
	
	Similarly, if $b=p^k[q^{m_1}-(-1)^{m_1},\dots,q^{m_s}-(-1)^{m_s}]$, where $p^{k-1}+1+m_1+\dots+m_s=n$ and $s\geq 2$, then $$b\leq p^k(q+1)q^{n-p^{k-1}-1-s}\leq (q+1)q^{n-1-s}\leq (q+1)q^{n-3},$$
	since $p^k\leq q^{p^{k-1}}$. 
	
	On the other hand, $n\leq (q+1)/2$ and $q\geq 9$, thus $a\geq \frac{2(q^n+1)}{(q+1)^2}-1>(q+1)q^{n-3}$.\end{proof}

Denote by $s(G)$ the number of connected components of the prime graph $GK(G)$, and by $\pi_i(G)$ the $i$-th connected component, $1\le i\le s(G)$. If the order of $G$ is even, then we assume that $2\in\pi_1(G)$. Let $\mu_i(G)$ (respectively, $\omega_i(G)$) be the set consisting of those $n\in\mu(G)$ ($n\in\omega(G)$) such that every prime divisor of $n$ belongs to $\pi_i(G)$.

The following statement is due to Gruenberg and Kegel, and was published  in \cite{81Wil}.

\begin{lemma}[Gruenberg–Kegel]\label{l:gk}
Let $G$ be a finite group and suppose that the prime graph $GK(G)$ is disconnected. Then one of the following holds:
\begin{enumerate}
\item[(a)] $G$ is a Frobenius group or a $2$-Frobenius group;
\item[(b)] $G$ is an extension of a nilpotent $\pi_1(G)$-group $N$ by a group $G_1$ with $S\le G_1\le\Aut(S)$, where $S$ is a nonabelian simple group and $G_1/S$ is a $\pi_1(G)$-group. Moreover, $s(S)\ge s(G)$ and for each $i$ with $2\le i\le s(G)$ there exists $j$ with $2\le j\le s(S)$ such that $\omega_i(G)=\omega_i(S)$. 
\end{enumerate}
\end{lemma}

\begin{lemma}\label{l:gk2}
Let $S$ be a finite simple group whose prime graph $GK(S)$ is disconnected. Then $|\mu_i(S)|=1$ for each $2\le i\le s(S)$. Denote by $n_i=n_i(S)$ the unique element of $\mu_i(S)$. Then $S$, $\pi_1(S)$, and the numbers $n_i$ for $2\le i\le s(S)$ are known.
\end{lemma}
\begin{proof}
See~\cite[Lemma 4]{00KonMaz.t}. 
\end{proof}

We consider only classical groups with disconnected prime graph.

\begin{table}[ht!]
 \renewcommand{\arraystretch}{1.2}   \captionsetup{justification=centering}
    \centering
    \caption{Simple classical groups with disconnected prime graph}
    \label{disconnected_graph_table}
    \begin{tabular}{|c|c|c|}
        \hline
        $S$ & \text{Conditions on $S$} & $n_i(S), 2\le i\le s(S)$ \\
        
        \hline
         $L_{2}(u)$ & $3<u\equiv e\pmod 4,\, e=\pm1$ & $\pi(u),(u+e)/2$ \\
        \hline
        $L_{2}(u)$ & $u>2,\, u\text{ is even}$ & $u-1, \, u+1$ \\
        \hline
        $L^\tau_r(u)$ & $(r,u,\tau)\neq (3,4,+)$ & \rule{0pt}{7mm}$\dfrac{u^{r}-\tau}{(u-\tau)(r,u-\tau)}$ \\[1.8ex]
        \hline
        $L^\tau_{r+1}(u)$ & $u-\tau~|~r+1$, $(r,u,\tau)\neq (5,2,-)$ & $(u^{r}-\tau)(u-\tau)$ \\
        \hline
       
        $L^+_3(4)$ & & $3, \,5, \,7$ \\
        \hline
        $L^-_6(2)$ & & $7,11$ \\
        \hline
        $S_{2m}(u)$ & $m=2^l\ge2$ & $(u^{m}+1)/(2,u-1)$ \\
        \hline
        $S_{2r}(u)$ & $u\in\{2,3\}$ &$(u^{r}-1)/(2,u-1)$ \\
        \hline
        $O_{2m+1}(u)$ & $m=2^l\ge4,\, u\text{ is odd}$ &  $(u^{m}+1)/(2,u+1)$ \\
        \hline
        $O_{2r+1}(3)$ & $r\geq3$ & $(3^{r}-1)/2$ \\
        \hline
        $O^+_{2r}(u)$ & $r\geq5,\, u\in\{2,3,5\}$ & $(u^{r}-1)/(u-1)$ \\
        \hline
        $O^+_{2r+2}(u)$ & $u\in\{2,3\}$ & $(u^{r}-1)/(2,u-1)$ \\
        \hline
        $O^-_{2m}(u)$ & $m=2^l\geq 4$ & $(u^{m}+1)/(2,u-1)$ \\
        \hline
        $O^-_{2m}(2)$ & $m = 2^l+1\geq 5$ & $2^{m-1}+1$ \\
        \hline
        $O^-_{2r}(3)$ & $5\leq r\neq2^l+1$ & $(3^{r}+1)/4$ \\
        \hline
        $O^-_{2m}(3)$ & $9\le m = 2^l+1\neq r$ & $(3^{m-1}+1)/2$ \\
        \hline
        $O^-_{2r}(3)$ & $r=2^l+1\geq5$ & $(3^{r-1}+1)/2$, $(3^{r}+1)/4$ \\
        \hline
    \end{tabular}
\end{table}

\begin{lemma}\label{disconnected_graph_prop}
Let $S$ be a classical group with disconnected prime graph. Then $S$ is a group in Table~$\operatorname{\ref{disconnected_graph_table}}$ (in the table, $r$ denotes an odd prime).	
\end{lemma}
\begin{proof}
See, for example,~\cite{02Maz.t}.
\end{proof}

A coclique of a graph is a set of pairwise non-adjacent vertices. The maximum size of a coclique is called the independence number of the graph. The independence number of the prime graph $GK(G)$ is denoted by $t(G)$. For $r\in\pi(G)$, the maximum size of a coclique of $GK(G)$ containing $r$ is called the $r$-independence number and is denoted by $t(r,G)$. 

The adjacency criteria in the prime graphs of simple groups, as well as the independence numbers of these graphs and some maximal cocliques, are described in \cite{05VasVd.t}. All maximal cocliques, along with corrected versions of some results from \cite{05VasVd.t}, can be found in \cite{11VasVd.t}.
 
 \begin{rem}
    According to Table 2 in \cite{11VasVd.t}, if $G={}^2D_n(q)$ with $n>4$ and $n\equiv 2\pmod 4$, then the maximum cocliques of $GK(G)$ are exactly the sets of the form $\Theta(G)\cup{r_n(q)}$, $\Theta(G)\cup \{r_{n/2}(q)\}$, and $\Theta(G)\cup \{r_{n-2}(q)\}$, where $\Theta(G)=\{r_i(q) \mid i \text{ even}, n<i\leq 2n \text{ or } i \text{ odd}, n<2i \leq 2n\}$. This is true in all cases except when $n=6$. In the case when $n=6$, since $2n/(n-2)$ is an odd integer, $r_{n-2}(q)$ is adjacent to $r_{2n}(q)$ by \cite[Proposition 2.5]{11VasVd.t}. Therefore, for $n=6$, the maximum cocliques of $GK(G)$ are $\Theta(G)\cup\{r_6(q)\}$ and $\Theta(G)\cup\{r_3(q)\}$, and in this case, $t(r_4(q),G)=4$. 
 \end{rem}

\begin{lemma}\label{l:str}
Let $L$ be a finite nonabelian simple group such that $t(L)\geq3$ and $t(2,L)\geq2$, let $G$ be a finite group such that $\omega(G)=\omega(L)$. Then the following statements hold.
\begin{enumerate}
\item[(a)] There is a nonabelian simple group $S$ such that $$S\leq \overline G=G/K\leq \Aut S,$$ where $K$ is the largest solvable subgroup of $G$.

\item[(b)] For every coclique $\rho$ of $GK(G)$ of size at least 3, at most one prime of $\rho$ divides  $|K|\cdot|\overline{G}/S|$. In particular, $t(S)\geq t(L)-1$.

\item[(c)] Every prime $r\in\pi(G)$ non-adjacent with $2$ in $GK(G)$ does not divide 
$|K|\cdot|\overline{G}/S|$. In particular, $t(2,S)\geq t(2,L)$.
\end{enumerate}
\end{lemma}
\begin{proof} See \cite{05Vas.t, 09VasGor.t}.
\end{proof}

\begin{lemma}\label{l:small} Under the hypothesis of Conjecture $\operatorname{1}$, we can assume that $S$ is different from $L_2(u)$, $S_4(u)$, and $S$ is not an exceptional group of Lie type.
\end{lemma}
\begin{proof}
See \cite[Proposition 3.8]{20GrZv.t} and \cite[Theorem 1]{24GrePan.t}.
\end{proof}

\section{Proof of Theorem 1}\label{s:general}

In this section, we prove Theorem~\ref{t:t1}. We split the proof into two cases:

\begin{enumerate}
    \item[1.] $L=U_n(q)$, where $n\in\{7,11,13,17,19,23\}$ and $q>3$;
    \item[2.] $L = L^\varepsilon_{n}(q)$, where $n\in\{8,12,14,18,20,24\}$ and $q-\varepsilon$ divides $n$, or $L = U_{n}(3)$, where $n\in\{7,11,13,17,19,23\}$. 
\end{enumerate}

The groups $U_n(3)$, where $n$ is a prime, are considered in Case 2, because some inequalities necessary for the proof are not satisfied for $q=3$. It should also be noted that the recognition problem for the groups $U_n(3)$ having disconnected prime graph was considered in \cite{16He}, however, the full text of this paper is not available to us, and its main result, at least in the form in which it is stated in databases, is incorrect: it claims that all such groups except $U_3(3)$ are recognizable, but $\omega(U_n(3))=\omega(\Aut (U_n(3)))$ for any prime $n$ by \cite[Theorem 1]{17Gre}.

Let $\omega(G) = \omega(L)$. Then $GK(G)$ is disconnected. By virtue of~\cite{03Ale.t} we can assume that $G$ is not a Frobenius group or a double Frobenius group. Thus, item (b) of Lemma~\ref{l:gk} is satisfied and the graph $GK(S)$ is also disconnected. Taking into account Lemma~\ref{l:small} we conclude that $S$ is one of the groups in Table~\ref{disconnected_graph_table} other than $L_2(u)$, $S_4(u)$. By Lemma~\ref{l:gk2} the equality $n_2(L) = n_i(S)$ holds for some $i$, $2 \le i \le s(S)$.

\medskip

\textbf{Case 1.} Let $L=U_n(q)$, where $n\in\{7,11,13,17,19,23\}$ and $q>3$. By Table~\ref{disconnected_graph_table}, we have $$n_2(L)=k_n(-q)=\frac{q^n+1}{(n,q+1)(q+1)}.$$

\begin{lemma}\label{no_symplectic_etc}
	$S\neq S_{2m}(u), O_{2m+1}(u), O^-_{2m}(u)$, where $m=2^l\ge4$.
\end{lemma}
\begin{proof}
	Assume the contrary. Then 
	$$
		\dfrac{q^n+1}{(n,q+1)(q+1)}=n_2(S)=\dfrac{u^m+1}{(2,u-1)}.
	$$
    
	Note that
	\begin{equation}\label{e:31}
		\dfrac{q^n+1}{(q+1)}-1=\dfrac{q(q^{n-1}-1)}{q+1}
	\end{equation}

        If $u$ is even, then $n_2(L) - 1 = u^m$. At the same time, from~\eqref{e:31} and item (a) of Lemma~\ref{MAGr_lemma} it follows that $n_2(L) - 1$ is not a prime power. Therefore, $u$ is odd and $n_2(L) - 1 = (u^m - 1)/2 \in \omega(S)$. This contradicts item (b) of Lemma~\ref{MAGr_lemma}.        
\end{proof}

\begin{lemma}\label{finite_cases}
        The following statements hold: 
        \begin{itemize}
            \item[(a)] $q<u^{m/(n-2)}$;
            \item[(b)] $\exp(L)<q^{F(n)+2}$.
        \end{itemize}
\end{lemma}

\begin{proof}
(a) We have $n_i(S)\leq u^m+1$. At the same time, $n_2(L)=\dfrac{q^n+1}{(q+1)(n,q+1)}>q^{n-2}+1$. From the equality $n_2(L)=n_i(S)$ we obtain that $q^{n-2}+1<u^m+1$.
    
(b)   Since $q>3$, we have $\exp_p(L)\leq q^2$. It can also be verified directly that $\prod_{i=1}^n \Phi_i(-q)<q^{F(n)}$ for $n\le 24$. Now the required estimate follows from Lemma~\ref{l:exponents}.
\end{proof}
\begin{lemma}\label{S_bounds}
	 The dimension of $S$ satisfies the constraints from Table~$\operatorname{\ref{bounds}}$.
		\begin{table}[ht]
        \caption{}\label{bounds}
            \centering
			$\begin{array}{|c|c|c|c|c|c|c|}
				\hline
				\text{\backslashbox{$L$}{$S$}} & L^\tau_m(u)& S_{2m}(u)& O_{2m+1}(3)& O^-_{2m}(u)&O^+_{2m}(u)&O^+_{2m+2}(u)  \\ \hline
				U_{7}(q) & 3\le m\le 12 & m=3,5,7& m=3,5,7& m=5,7 &m=5,7 &m=3,5,7 \\ \hline
				U_{11}(q) & 11\le m\le 14 &m=5,7& m=5,7&m=7,9&m=7&m=7  \\ \hline
				U_{13}(q) & 11\le m\le 14 &m=7& m=7&m=7,9,11&m=11&-  \\ \hline
				U_{17}(q)& 17\le m\le 20 &m=11& m=11&m=11,13&m=13&m=11  \\ \hline
				U_{19}(q) & 17\le m\le 20 &m=11,13& m=11,13&m=11,13&m=13&m=13  \\ \hline
				U_{23}(q)& 23\le m\le 24 &m= 13& m= 13 &m=17&m=17& -  \\ \hline
			\end{array}$		  			  
		\end{table}
\end{lemma}
\begin{proof}
Note that by virtue of Lemma~\ref{no_symplectic_etc} and Table~\ref{disconnected_graph_table}, if $S=S_{2m}(u)$, $O_{2m+1}(u)$, or $O_{2m}^-(u)$, then $m$ is odd.

By item (c) of Lemma~\ref{l:str}, we have $t(S)\geq t(L)-1$. Now using the independence numbers of the prime graphs of simple groups from \cite{05VasVd.t, 11VasVd.t}, it is easy to derive a lower bound on the dimension of $S$.

Let us prove the upper bound. Clearly, $u^b=\exp(S)\le\ \exp(G)=q^{a}$, where $a=\log_q(\exp(L))$ and $b= \log_u(\exp(S))$. Hence $u^{b/a}<q$. Applying Lemma~\ref{finite_cases}, we see that $a<F(n)+2$ and $m/(n-2)>b/a>b/(F(n)+2)$. Therefore,
\begin{equation}\label{e:l32} b/m<(F(n)+2)/(n-2). \end{equation}
In virtue of Lemma~\ref{l:exponents}, the ratio $b/m$ tends to infinity as $m$ tends to infinity, therefore $m$ is bounded. For each group $L$ we find the values of $m$ that satisfy the inequality~\eqref{e:l32}.

Let $L=U_{23}(q)$ and $S=S_{2m}(u), O_{2m+1}(u)$. Then by Lemma~\ref{l:exponents}, we have $b>3F(m)/2$. Hence, $3F(m)/2m<174/21$. It follows that $m\le16$. By Table~\ref{disconnected_graph_table} and Lemma~\ref{no_symplectic_etc}, we obtain that $m=13$. Other cases when $S$ is a symplectic or orthogonal group are handled in the same manner.

Let $L=U_{11}(q)$ and $S=L_m^\tau(u)$. Then $b>3F(m)/4$ by Lemma~\ref{l:exponents} and, hence, $3F(m)/4m<44/9$. It follows that $m\leq 20$. For $m\in\{17,18,19,20\}$ we use a more precise estimate from Lemma~\ref{l:23} and derive that $(F(m)-2)/4<44/9$. Thus, $m\le 14$. Other cases when $S$ is a linear or unitary group are considered in the same manner.
\end{proof}
\begin{lemma}
	$S\neq L^\tau_{m+1}(u),\, S_{2m}(u),\, O_{2m+1}(u),\, O^\tau_{2m}(u)$, where $m$ is a prime.
\end{lemma}
\begin{proof}

Assume first that $S=S_{2m}(u),\, O_{2m+1}(u),\, O^\tau_{2m}(u)$. Since $GK(S)$ is disconnected, Table 1 and Lemma 3.3 imply that for each $n$ it is necessary to consider only finitely many possibilities for $S$. Denote by $\Sigma$ the set of the numbers $n_i(S)$, $i\ge 2$, for these $S$.

The values $n_i(S)$ for all possible groups $S$ are presented in Table~\ref{ks}. Note that $s(S)=2$ if $S\neq O^{-}_{10}(3)$, and $n_2(S)$ for different values of $u$ are separated by a semicolon. If $S=O^{-}_{10}(3)$, then $s(S)=3$, $n_2(S)=41$, $n_3(S)=61$.
 
\begin{table}[htb!]
	\renewcommand{\arraystretch}{1.3}
        \centering
        \captionsetup{justification=centering}
        \caption{}\label{ks}
			$\begin{array}{|c|c|c|}
				\hline
				S & \text{Conditions}& n_i(S), 2\le i\le s(S) \\ \hline
				S_{2m}(u), O^+_{2m+2}(u) & m=3, u=2,3 & 7;13 \\
                & m=5, u=2,3&31;121\\
                & m=7, u=2,3&127;1093\\
                & m=11, u=2,3&2047;88573\\
                & m=13, u=2,3&8191;797161\\ \hline
                O_{2m+1}(3) & m=5,7,11,13& 121;1093;88573;797161 \\\hline
                O^-_{2m}(u) & m=5, u=2,3& 17;41,61 \\
                & m=7, u=3&547\\
                 & m=9, u=2,3&257;3281 \\
                & m=11, u=3&44287\\
                & m=13, u=3&398581\\
                & m=17, u=2,3&65537;21523361,32285041\\\hline
                O^+_{2m}(u) & m=5, u=2,3,5& 31;121;781 \\
                & m=7, u=2,3,5&127;1093;19531\\
                & m=11, u=2,3,5&2047;88573;12207031\\
                & m=13, u=2,3,5&8191;797161;305175781\\
                & m=17, u=2,3,5&131071;64570081;190734863281\\\hline
			\end{array}$		  			  
\end{table}    
    
Let $n=7$. Then $m\le7$ and $\Sigma={7,13,17, 31, 41, 61, 121, 127, 547, 781, 1093, 19531}$. If $q\geq7$, then $n_2(L)>\max(\Sigma)$, and if $q=3,5$, then $n_2(L)=547,13021$. Hence, $n_i(S)=547$ and $S=O_{14}^-(3)$. This contradicts to the fact that the defining characteristics of $L$ and $S$ are different.

Let $n=11,13$. Then $\Sigma=\{7,13,17, 31, 41, 61, 121, 127, 547, 781, 1093, 19531\}$. If $q\geq7$, then $n_2(L)>\max(\Sigma)$. If $q=3,5$ and $n_2(L)\in \Sigma$, then $q=u=3$; a contradiction.

Similarly, if $n=17,19,23$ and $q\ge 5$, then $n_2(L)>\max(\Sigma)$, and if $q=3$, then $n_2(L)\not\in \Sigma$.

Let $S=L^\tau_{m+1}(u)$. Then    
	
	\begin{equation*}
		\dfrac{q^n+1}{(n,q+1)(q+1)}=n_2(S)=\dfrac{u^{m-1}-\tau}{u-\tau}.
	\end{equation*}
	Assume that $(n,q+1)=1$. Subtracting 1 from both sides of the above equality, we obtain
        \begin{equation*}
        \dfrac{q(q^{n-1}-1)}{q+1} = \dfrac{u(u^{m-2}-1)}{u-\tau}.
        \end{equation*} As already noted in the proof of Lemma~\ref{MAGr_lemma}, the number $pr_{n-1}(-q)$ is not in $\omega(G)$. On the other hand, by Lemma~\ref{L_spectra} any two prime divisors of the expression on the right hand side are adjacent in $GK(S)$; a contradiction. Consequently, $n$ divides $q+1$, in particular, $q\ge2n-1$.
        
        Now we proceed as before, by writing for each $n$ the set $\Sigma$ consisting of the $n_i(S)$ for all possible $S$.	
		
	Let $n=7$. Then $S$ is one of the groups $L_{12}(u), u\in\{13,7,5,4,3,2\}$; $U_{12}(u), u\in\{11,5,3,2\}$; $S=L_8(u), u\in\{9,5,3,2\}$; $U_8(u), u\in\{7,3\}$; $L_6(u), u\in\{7,4,3,2\}$; $U_6(u), u\in\{5,2\}$; $L_4(u), u\in\{5,3\}$; $U_4(3)$; and
\begin{equation*}
		\begin{aligned}
			\Sigma& = \{7, 11, 13, 31, 121, 127, 341, 521, 547, 683, 1093, 2047, 2801, 19531,44287,\\& 88573, 102943, 597871, 1398101, 8138021, 12207031, 329554457, 149346699503\}.
	\end{aligned}
\end{equation*}
	The cases $n_i(S)\in\{1398101, 8138021, 12207031, 329554457, 149346699503\}$ are impossible due to Lemma~\ref{ri_structure}. For $q>9$ we have $n_2(L)>597871$. The case $q\leq 9$ is impossible since $q\ge 2n-1$.

	Let $n=11, 13$. Then $S=L_{12}(u), u\in\{13,7,5,4,3,2\}$; $U_{12}(u), u\in\{11,5,3,2\}$; $L_{14}(u), u\in\{8,3,2\}$; $U_{14}(13)$; and
	\begin{equation*}
		\begin{aligned}
			\Sigma& = \{683, 2047, 8191, 44287, 88573, 1398101, 8138021, 12207031, \\
			& 329554457, 23775972551, 48444505203,  78536544841, 149346699503, 21633936185161\}.
		\end{aligned}
	\end{equation*}

	If $n=11$ and $q>19$, then $n_2(L)>\max(\Sigma)$. A similar inequality holds for $n=13$ and $q>13$.
    
	If $n=17,19$, then $S=L_{20}(u), u\in\{11,5,3,2\}$; $U_{20}(u), u\in\{19, 9, 4, 3\}$; $L_{18}(u), u\in\{19,7, 4,3,2\}$; $U_{18}(u), u\in\{17,8,2\}$. If $n=23$, then $S=L_{24}(u), u\in\{25,13,9,7,5,4,3,2\}$; $U_{24}(u), u\in\{23, 11,7,5,3,2\}$. As before, we find the set $\Sigma$ and verify that $n_2(L)>\max(\Sigma)$ when $q\ge 2n-1$.
\end{proof}

\begin{lemma}\label{K_structure}
	Let $r\in\pi(K)$. If $S = L^\tau_m(u)$, where $m\geq3$ is odd, then $t(r,G)\leq2$ and $r$ divides $q^2-1$.
\end{lemma}
\begin{proof}
        By~\cite[Propositions 2.1, 2.2]{05VasVd.t}, for any $s\in \pi(G)\setminus \pi(q^2-1)$, we have $t(s,G)\geq3$. The remaining part of the proof repeats the proof of~\cite[Lemma 5.5]{24GrePan.t}.
\end{proof}

\begin{lemma}\label{l:linear_unitary}
	$S\neq L^\tau_m(u)$, where $m$ is a prime.
\end{lemma}
\begin{proof}	
        Assume the contrary. Then $n_2(L)=n_2(S) = k_m(\tau u)$. Subtracting 1 from both sides of this equality, multiplying by $(m,u-\tau)$, and setting $$t=(u^m-\tau)/(u-\tau),$$ we get

	\begin{equation}\label{e:main}
		(m,u-\tau)\dfrac{q(q^{n-1}-1)}{q+1}=(n,q+1)\cdot t - (m,u-\tau).
	\end{equation}

Assume that $(n,q+1)\neq (m,u-\tau)$. From \eqref{e:main} it follows that the number $$b=(n,q+1)\cdot t - (m, u-\tau)$$  is divisible by $k_{n-1}(-q)$. By Lemma~\ref{K_structure}, we have $R_{n-1}(-q) \cap \pi(K)=\varnothing$. Moreover, $R_{n-1}(-q) \cap \pi(\overline{G}/S)=\varnothing$ by~\cite[Lemma 3.4]{20GrZv.t}. Therefore, $k_{n-1}(-q)\in\omega(S)$.

Let $v$ denote the characteristic of the defining field of $S$. Suppose that $v$ divides $k_{n-1}(-q)$. We have $t\equiv1\pmod v$ and hence $$b\equiv (n,q+1)-(m,u-\tau)\equiv 0 \pmod v.$$ Consequently, $v$ divides $(n,q+1)-(m,u-\tau)$. By going through all the possible values of this difference, we see that $v\le 3$ when $n<23$ and $v\le 11$ when $n=23$.
In addition, $v\equiv 1\pmod {n-1}$ by Lemma~\ref{ri_structure}; a contradiction.

        Hence, $v$ does not divide $k_{n-1}(-q)$. Suppose that a prime $r$ divides $(k_{n-1}(-q), u-\tau)$. By Lemma~\ref{l:str}, the inequality $t(r,S)\geq t(r,L)-1\geq 3$ holds. On the other hand, if $t(r,S)\ge3$, then by~\cite[Propositions 4.1, 4.2]{05VasVd.t}, we have $t(r,S)=3$ and $m_r=(u-\tau)_r\ge r$. Hence $n=7$ and, therefore, $m\le 12$. Since $r_6(-q)\ge 7$, we conclude that $r=7$, $m=7$, and $(k_6(-q))_7=(\exp(S))_7=7$ by Lemmas \ref{l:r-part} and \ref{L_spectra}. Thus, either $(k_{n-1}(-q), u-\tau)=1$, or $n=m=7$, $7\in R_6(-q)$, $(u-\tau)_7=7$ and $(k_{6}(-q)/7, u-\tau)=1$. Set $k=k_{6}(-q)/7$ in the latter case and $k=k_{n-1}(-q)$ in the remaining cases.     

    Since $k\in\omega(S)$ and $(k,v(u-\tau))=1$, by Lemma~\ref{L_spectra} there exist $m_1,\dots, m_s\geq 2$ such that $m_1+\dots+m_s\leq m$ and $k$ divides $[u^{m_1}-\tau^{m_1},\dots, u^{m_s}-\tau^{m_s}]/(u-\tau)$. As $$\frac{u^l-\tau^l}{u-\tau}=\prod_{1<e\mid l}|\Phi_e(\tau u)|,$$ it follows that $(b,(u^l-\tau^l)/(u-\tau))$ divides $\prod_{1<e\mid l}(b,\Phi_e(\tau u))$. Thus, $k$ divides $[b_{m_1},\dots,b_{m_s}]$, where $b_l=\prod_{1<e\mid l}(b,\Phi_e(\tau u))$.

    We will show that $b$ and $\Phi_e(\tau u)$, where $e>1$, considered as polynomials in $\mathbb Q[u]$, are coprime. Consequently, there exist $f,g\in \mathbb Z[u]$ and $c\in \mathbb Z$ such that $fb+g\Phi_e(\tau u)=c$. For any integer $u$, the greatest common divisor of the numbers $b$ and $\Phi_e(\tau u)$ divides $c$. We refer to any such number $c$ as an integral greatest common divisor of $b$ and $\Phi_e(\tau u)$.
    
    All computations below were performed with the aid of a computer (see the corresponding functions in Appendix~\ref{lst:python_example}).

 1.  For the polynomials $b$ and $\Phi_e(\tau u)$, where $1<e<m$, we compute an integral greatest common divisor $c_e$ using the function \textit{integer\textunderscore polynomial\textunderscore gcd} (this function is based on Steps 1 and 2 of Algorithm 1.2 from~\cite{23Lub}).
 
 2. Using the function \textit{prime\textunderscore factorization\textunderscore string}, we obtain the prime power factorization of $c_e$. By Lemma~\ref{ri_structure}, a necessary condition for a prime to belong to $R_{n-1}(-q)$ is its congruence to $1$ modulo $n-1$. Using the function \textit{primitive\textunderscore candidates\textunderscore prime\textunderscore factors}, we find the subset $A_{n-1}$ of $\pi(c_e)$ consisting of primes satisfying this condition and we set $a_e=({c_e})_{A_{n-1}}$.

 3. For all tuples $m_1,\dots,m_s$ such that $m_1+\dots+m_s\leq m$, we compute the numbers $[a_{m_1},\dots,a_{m_s}]$. Denote the set of such numbers by $\Sigma$. Since $k$ divides some number in $\Sigma$, we obtain a series of equations on $q$. Then we verify that none of these equations has solutions in prime powers.

  Since $(n,q+1)\neq (m,u-\tau)$, three cases arise: $b=t-m$, $b=nt-m$, where $n\neq m$, and $b=nt-1$. The computed sets $\Sigma$ for all three cases are given in Table~\ref{tab:dis2}.

 \captionsetup[table]{justification=centering}
\begin{table}[ht!]
\caption{}\label{tab:dis2}
\begin{tabular}{|c|c|c|l|}

\hline
$n$ & $m$ & $b$ & $\Sigma$ \\
\hline

23 & 23 & $t-m$ & $\{6689, 3301\cdot 139921, 72161\cdot 139921, 531690051726901, 12271836836138419\}$ \\
   &    & $nt-1$ & $\{245411, 3301\cdot 139921\}$ \\
\hline
19 & 19 & $t-m$ & $\{181\cdot8461, 181\cdot 30170017, 811\cdot 2251, 127\cdot301267, 6733, 19102029337\}$ \\
   &    & $nt-1$ & $6733, 181\cdot8461, 2251\cdot811, 127\cdot301267, 181\cdot30170017\}$ \\
   & 17 & $t-m$ & $\{3907, 8191, 307\cdot1117, 1423\cdot5653\}$ \\
   &    & $nt-m$ & $\{109\cdot1153\cdot 10837, 127\cdot883\}$ \\
   &    & $nt-1$ & $\{117991\cdot14643661951, 127\cdot37, 3529\cdot379\cdot2143, 199, 1911519592511491\}$ \\
\hline
17 & 19 & $t-m$ & $\{17, 6449\cdot854417\}$ \\
   &    & $nt-m$ & $\{2161, 593\cdot19082929, 34273\cdot 6732717349681\}$ \\
   &    & $nt-1$ & $\{1633816708129, 3361, 2113\cdot 1704449\}$ \\
   & 17 & $t-m$ & $\varnothing$ \\
   &    & $nt-1$ & $\{1009, 2113\}$ \\
\hline
13 & 13 & $t-m$ & $\{61, 3697, 440677, 266981089\}$ \\
   &    & $nt-1$ & $\{22621, 101406750589\}$ \\
   & 11 & $t-m$ & $\{37, 61\}$ \\
   &    & $nt-m$ & $\{37, 433\}$ \\
   &    & $nt-1$ & $\{61\cdot157, 61\cdot 4016377\}$ \\
\hline
11 & 13 & $t-m$ & $\{11^3\cdot41\cdot31\cdot 61\}$ \\
   &    & $nt-m$ & $\{4951\cdot86531\}$ \\
   &    & $nt-1$ & $\{41\cdot71\cdot271\cdot3331, 71\cdot1171, 71\cdot 1623931\}$ \\
   & 11 & $t-m$ & $\{61\}$ \\
   &    & $nt-1$ & $\{61\}$ \\
\hline
7 & 11 & $t-m$ & $\{7,61,37\cdot333667,37\cdot7\cdot 19\}$ \\
  &    & $nt-m$ & $\{31\cdot13\cdot 19, 31\cdot19\cdot2521,1300237\}$ \\
  &    & $nt-1$ & $\{19^2\cdot43,55987,19\cdot43\cdot90217\}$ \\
  & 7  & $t-m$ & $\varnothing$ \\
  &    & $nt-1$ & $\{31\cdot61\}$ \\
  & 5  & $t-m$ & $\{7\}$ \\
  &    & $nt-m$ & $\{13\}$ \\
  &    & $nt-1$ & $\{43\}$ \\
  & 3  & $t-m$ & $\varnothing$ \\
  &    & $nt-m$ & $\varnothing$ \\
  &    & $nt-1$ & $\varnothing$ \\
\hline
\end{tabular}
\end{table} 
  
  As an example, consider the case $n=23$ and $b=t-m$. In this case $$\Sigma =\{6689,3301\cdot 139921,72161\cdot 139921, 531690051726901, 12271836836138419\}.$$ Using computer calculations, we verify that any divisor of any number in this set cannot be equal to  $(q^{23}+1)/((23,q+1)(q+1))$.

        Consequently, we can assume that $(n,q+1)=(m,u-\tau)$. Subtract $1$ from both sides of the equality $n_2(L)=n_2(S)$ again. We obtain 
	\begin{equation}\label{km=kn}
		\dfrac{q(q^{n-1}-1)}{q+1} = \dfrac{u(u^{m-1}-1)}{u-\tau}.
	\end{equation}
    
	By Lemma~\ref{L_spectra}, the primes $p$ and $r_{n-1}(-q)$ are not adjacent in $GK(G)$. On the other hand, any pair of distinct prime divisors of the right-hand side of~\eqref{km=kn} whose product does not lie in $\omega(S)$ has the form $v$ and $r_{m-1}(\tau u)$. Hence, $p\in R_{m-1}(\tau u)$ and $k_{n-1}(-q)=u$. Substituting the last equality into~\eqref{km=kn}, we have	
	\begin{equation}\label{k_equation}
		\dfrac{q(q^{n-1}-1)}{q+1} = \dfrac{k_{n-1}(-q)(k_{n-1}(-q)^{m-1}-1)}{k_{n-1}(-q)-\tau}.
	\end{equation}

        With the aid of computer, it is easy to verify that for any relevant pair $n,m$, the equation~\eqref{k_equation} in the variable $q$ has no primary roots.
\end{proof}

Thus, in Case 1 the proof is complete.

\medskip

\textbf{Case 2.} The possible $n,q$ are listed in Table~\ref{tab:dis0}, and the corresponding value $n_2(L)$ is indicated there as well.

Let $S=L_m^\tau(u)$ with $m$ prime. Then $$n_2(S)\ge \frac{u^m+1}{m(u+1)}\ge \frac{2^m+1}{3m}.$$ Hence, if $m>107$, then $n_2(L)<n_2(S)$. For each prime $m\le 107$, computer calculations show that the equations $n_2(L)= n_2(S)$ have no solutions in prime powers.

By analyzing in a similar manner all the remaining possibilities for the group $S$, we handle Case 2, and thus complete the proof of Theorem 1.

\newpage
   \begin{table}[ht!]

\renewcommand{\arraystretch}{1.3}
\centering
\captionsetup{justification=centering}

\caption{}\label{tab:dis0}
    $\begin{array}{|c|c|c|c|}
        \hline
        n &\varepsilon & q& n_2(L)  \\ \hline
        7 & - & 3 & 547 \\ \hline
        8 & + & 3,5,9 & 1093,\,19531,\,547\cdot1093\\
        & -&3,7&547,\,113\cdot911
        \\ \hline
        11 & - & 3 & 67\cdot661 \\ \hline
        12 & + & 3,5,7,13 & 23\cdot3851,\,12207031,\,1123\cdot 293459,\, 23\cdot419\cdot859\cdot18041\\
        & -&3,5,11&67\cdot661,\, 23\cdot67\cdot5281,\,23\cdot 89\cdot 199\cdot 58367
        \\ \hline
        13 & - & 3 & 398581 \\ \hline
        14 & + & 3 & 797161\\
        & -&13&13417\cdot 20333\cdot 79301
        \\ \hline
        17 & - & 3 & 103\cdot 307\cdot 1021 \\ \hline
        18 & + & 3,7,19 & 1871\cdot 34511,\, 14009\cdot 2767631689,\,3044803\cdot99995282631947\\
        & -&5,17&3061\cdot41540861,\, 45957792327018709121
        \\ \hline
        19 & - & 3 & 2851\cdot 101917 \\ \hline
        20 & + & 3,5,11 & 1597 \cdot 363889,\, 191\cdot 6271\cdot 3981071,\,6115909044841454629 \\
        & -&3,9,19&2851\cdot 101917,\, 5301533\cdot 25480398173,\,108301\cdot 1049219\cdot 870542161121

        \\ \hline
        23 & - & 3 & 47\cdot 1001523179 \\ \hline
        24 & + & 3,5,7 & 47\cdot 1001523179, 8971\cdot 332207361361,47\cdot 3083\cdot 31479823396757,  \\
        & &9,13 &47\cdot 1001523179\cdot 23535794707,\, 1381\cdot 2519545342349331183143 \\
        & &25 &47\cdot 8971\cdot 332207361361\cdot 42272797713043 \\
        & -&3,5,7&23535794707, 47\cdot 42272797713043, 3421093417510114543\\
        & &11,23 &47\cdot 1013\cdot 241363\cdot 6493405343627
        \\
        & &23 & 47\cdot 139\cdot 1013\cdot 1641281\cdot 52626071\cdot 1522029233
        \\
        \hline
    \end{array}$		  			  
\end{table} 

\newpage
\section{Appendix}

\lstset{
    language=Python,
    basicstyle=\ttfamily\small, 
    keywordstyle=\color{blue}, 
    commentstyle=\color{red}, 
    stringstyle=\color{red}, 
    frame=single, 
    breaklines=true, 
    tabsize=2, 
    numbers=none
}

\begin{lstlisting}[
    label={lst:python_example},
    caption={Utility functions in Python},
    captionpos=t
]
from sympy import (
    symbols, cyclotomic_poly, expand, gcdex, lcm, Poly, gcd, 
    simplify, cancel, factorint)
u = symbols('u')

def get_cyclotomic_polynomial(n: int, tau: int) -> Poly:
    """
    Compute the n-th cyclotomic polynomial.
    """
    return cyclotomic_poly(n, tau * u)

def cyclotomic_product(n_terms: int, tau: int) -> Poly:
    """
    Calculate the product of the first n cyclotomic polynomials.
    """
    product = 1
    for order in range(1, n_terms + 1):
        product *= cyclotomic_poly(order, tau*u)
    return expand(product)

def primitive_candidates_prime_factors(n: int, base: int) -> list:
    """
    Extract primitive prime factor candidates of n.
    """
    prime_factors = factorint(n)
    return [p for p in prime_factors if (p - 1) % base == 0]

def prime_factorization_string(n: int) -> str:
    """
    Prime power representation of n.
    """
    factors = factorint(n)
    return " . ".join(f"{p}^{e}" for p, e in factors.items())
    
def integer_polynomial_gcd(poly1: Poly, poly2: Poly) -> tuple[int, Poly]:
    """
    Compute content and gcd of polynomials using extended Euclidian 
    algorithm. To be able to use non univariate polynomial notation, 
    one should pass  cancel(simplify(f)) as an argument.
    """
    a_poly, b_poly, _ = gcdex(poly1, poly2)
    a_coeffs = Poly(a_poly, u).coeffs()
    b_coeffs = Poly(b_poly, u).coeffs()
    denominators = [
        coeff.as_numer_denom()[1] 
        for coeff in a_coeffs + b_coeffs]    ]
    return (lcm(denominators), gcd(f_poly, g_poly))
\end{lstlisting}

\lstdefinelanguage{GAP}{
    morekeywords={function, local, end, return, for, do, od, if, then, else, fi, while, repeat, until},
    morecomment=[l]{\#},
    morestring=[b]",
    basicstyle=\ttfamily\small,
    keywordstyle=\bfseries\color{black},
    commentstyle=\color{gray},
    stringstyle=\color{red},
    showstringspaces=false,
    breaklines=true,
    frame=single,
    tabsize=4,
    captionpos=t
}

\end{document}